\documentclass{amsart2010}
\usepackage{graphicx}
\usepackage{amsmath}
\usepackage{amsfonts}
\usepackage{amssymb}

\newtheorem{thm}{Theorem}
\newtheorem{cor}[thm]{Corollary}
\newtheorem{lem}[thm]{Lemma}
\newtheorem{prop}[thm]{Proposition}
\theoremstyle{definition}

\newtheorem*{theorem*}{Theorem}

\theoremstyle{remark}

\numberwithin{equation}{section}
\def\<{\langle}
\def\>{\rangle}
\begin{document}
\title[]{A weighted composition semigroup related to three open problems}
\author{Juan Manzur, Waleed Noor, and Charles F. Santos}%
\address{IMECC, Universidade Estadual de Campinas, Campinas-SP, Brazil.}
\email{$\mathrm{juan\_manzur123@hotmail.com}$ (Juan Manzur)} 
\email{$\mathrm{waleed@unicamp.br}$ (Waleed Noor) Corresponding Author}
\email{$\mathrm{ch.charlesfsantos@gmail.com}$ (Charles F. Santos)}
\begin{abstract} The semi-group of weighted composition operators $(W_n)_{n\geq 1}$ where 
	\[
	W_nf(z)=(1+z+\ldots+z^{n-1})f(z^n)
	\]
	on the classical Hardy-Hilbert space $H^2$ of the open unit disk is related to the Riemann Hypothesis (RH) (see \cite{Waleed}). The semigroup $(W_n)_{n\geq 1}$ is also closely related to the Invariant Subspace Problem (ISP) and the Periodic Dilation Completeness Problem (PDCP). We obtain results on cyclic vectors, spectra, invariant and reducing  subspaces. In particular, we show that several basic questions related to the semigroup $(W_n)_{n\geq 1}$ are equivalent to the RH and provide generalizations of the Báez-Duarte criterion for the RH (see \cite{Baez-Duarte}).
\end{abstract}

{\subjclass[2010]{Primary; Secondary}}
\maketitle{}

\section{Introduction}
The Riemann Hypothesis (RH) is considered to be the most important unsolved problem in mathematics. In 1950, Nyman and Beurling (see \cite{Beurling, Nyman}) gave a functional analytic reformulation for the RH. They proved that the RH holds if and only if the function $\chi_{(0,1)}$ belongs to the closed linear span of $\{ f_{\lambda}: \lambda\in [0,1]\}$ in $L^{2}(0,1)$, where $f_{\lambda}(x)=\{\lambda/x\rbrace-\lambda\lbrace 1/x\}$. Here $\lbrace x\rbrace$ denotes the fractional part of a real number $x$. In 2003, Báez-Duarte \cite{Baez-Duarte} proved that the uncountable family $\lbrace f_{\lambda}: \lambda\in [0,1]\rbrace$ may be replaced by  the sequence $\lbrace f_{1/k}: k\geq1\rbrace$. Recently, Noor \cite{Waleed} gave an $H^{2}$ version of the Báez-Duarte Theorem.

\begin{thm}\label{waleed}(See \cite[Theorem 6]{Waleed})
For each $k\geq2$, define 
\begin{align*}
h_{k}(z)=\dfrac{1}{1-z}\log\left(\dfrac{1+z+\dots+z^{k-1}}{k}\right)
\end{align*}
and let $\mathcal{N}:=\mathrm{span}\{h_k:k\geq 2\}$.
Then the RH holds if and only if the constant $1$ belongs to the closure of $\mathcal{N}$ in $H^{2}$.
\end{thm}
In \cite{Waleed}, a semigroup of weighted composition operators $\mathcal{W}=(W_n)_{n\in\mathbb{N}}$ on $H^{2}$ was introduced, where 
\begin{equation}\label{W_n}
W_nf(z)=(1+z+\ldots+z^{n-1})f(z^n)=\frac{1-z^n}{1-z}f(z^n).
\end{equation}
Each $W_n$ is bounded on $H^2$, $W_1=I$ and $W_mW_n=W_{mn}$ for each $m,n\geq 1$. The connection of $\mathcal{W}$ to the RH stems from the fact that the linear manifold $\mathcal{N}$ is invariant under $\mathcal{W}$, that is, $W_n(\mathcal{N})\subset\mathcal{N}$ for all $n\geq 2$ (see \cite[page 249]{Waleed}). A vector $f\in H^2$ is called a \emph{cyclic vector} for an operator semigroup $\{S_n:n\geq 1\}$ if $\mathrm{span}\{S_nf:n\geq 1\}$ is dense in $H^2$. Since the constant $1$ is a cyclic vector for $\mathcal{W}$, the following generalization of Theorem \ref{waleed} was obtained.
\begin{thm}\label{Waleedsemigroup}(See \cite[Theorem 8]{Waleed})
The following statements are equivalent
\begin{enumerate}
	\item The RH holds true,
	\item the closure of $\mathcal{N}$ contains a cyclic vector for $\mathcal{W}$, 
	\item $\mathcal{N}$ is dense in $H^2$.
\end{enumerate}	
\end{thm}
The semigroup $\mathcal{W}$ is also related to an open problem in harmonic analysis known as the \emph{Periodic Dilation Completeness Problem} (PDCP). The PDCP asks which $2$-periodic functions $\phi$ on $(0,\infty)$ have the property that 
\[
\mathrm{span}\{\phi(nx):n\geq 1\}
\]
is dense in $L^2(0,1)$. Such $\phi$ are called \emph{PDCP} functions. This difficult open problem was first considered independently by Wintner \cite{Wintner} and Beurling \cite{Beurling 2}. See \cite{Hedenmalm-Seip}, \cite{Hui-Guo} and \cite{Nikolski} for beautiful modern treatments of the PDCP. The PDCP has an equivalent reformulation in $H^2$. That is, characterize the cyclic vectors for the semigroup $\mathcal{T}=(T_n)_{n\geq 1}$ on 
\[
H^{2}_{0}:=H^{2}\ominus\mathbb{C}=\{f\in H^2: f(0)=0\}
\]
where $T_{n}f(z)=f(z^{n})$  (see \cite{Nikolski}). The semigroups $\mathcal{T}$ and $\mathcal{W}$ are \emph{semiconjugate} in the sense that 
 \begin{equation}\label{semiconjugate}
 T_{n}(I-S)=(I-S)W_{n}
 \end{equation}
 where $S$ is the unilateral shift on $H^{2}$, and $I-S$ is injective with dense range (\cite{Waleed}). The relation \eqref{semiconjugate} implies that the cyclic vectors for $\mathcal{W}$ are properly embedded into the cyclic vectors for $\mathcal{T}$, and therefore into the collection of all PDCP functions (see \cite[Theorem 13]{Waleed}). 
 
 The third open problem that the semigroup $\mathcal{W}$ is related to is the \emph{Invariant Subspace Problem} (ISP), which asks: Does every bounded linear operator $T$ on a separable Hilbert space $\mathcal{H}$ have a nontrivial closed invariant subspace? An invariant subspace of $T$ is a subspace $M\subset\mathcal{H}$ such that $T(M)\subseteq M$. The recent monograph by Chalendar and Partington \cite{Chalendar Partington book} is a reference for some modern approaches to the ISP. Rota \cite{Rota} demonstrated the existence of operators that have an invariant subspace structure so rich that they could model \emph{every} Hilbert space operator. \\ \\
 $\mathbf{Definition.}$
 \emph{Let $\mathcal{\mathcal{H}}$ be a Hilbert space and $U$ a bounded linear operator on $\mathcal{H}$. Then $U$ is said to be universal for $\mathcal{H}$, if for any bounded linear operator $T$ on $\mathcal{H}$ there exists a constant $\alpha\neq 0$ and an invariant subspace $\mathcal{M}$ for $U$ such that the restriction $U|_\mathcal{M}$ is similar to $\alpha T$. } \\ 
 
So if $U$ is universal for a separable, infinite dimensional Hilbert space $\mathcal{H}$, then the ISP is equivalent to the assertion that every \emph{minimal} invariant subspace for $U$ is one dimensional. The main tool for identifying universal operators has been the following criterion of Caradus \cite{Caradus}. \\ \\
 $\mathbf{The \ Caradus \  Criterion}$ (see \cite{Caradus}.) \emph{Let $U$ be a bounded linear operator on $\mathcal{H}$. If $\mathrm{ker}(U)$ is infinite dimensional and $U$ is surjective, then $U$ is universal for $\mathcal{H}$. } \\
 
The plan of the paper is the following. After a section of preliminaries where we introduce the Hardy-Hilbert space $H^2$ and local Dirichlet spaces $\mathcal{D}_{\delta_\zeta}$, in Section 3 we begin the study of the semigroup of weighted composition operators $(W_n)_{n\in\mathbb{N}}$. We first state without proof an adjoint fomula for $W_n$ for each $n\geq 2$ (Lemma \ref{W_n^*}). As a consequence $W_n/\sqrt{n}$ is a shift of infinite multiplicity, and the adjoints $W_n^*$ for $n\geq 2$ are universal in the sense of Rota (Proposition \ref{Wn shift}). In Section 4, we first show that the shift-invariance of the closure of $\mathcal{N}$ is equivalent to the RH (Proposition \ref{Beurling RH}). We then study the lattice $\mathrm{Lat}(\mathcal{W})$ of common invariant subspaces of the $W_n$. In particular, the closure of the subspace $\mathcal{M}:=\mathrm{span}\{h_k-h_\ell:k,\ell\geq 2\}$ is shown to be a maximal invariant subspace if and only if the RH holds (see Theorem \ref{4}). In Section 5, we ask whether the closure of $\mathcal{N}$ is a reducing subspace for any $W_n$  and show that this question is also equivalent to the RH (Theorem \ref{newreform}). In Section 6, we prove that polynomials $z^m+\ldots+z-\lambda$ for $|\lambda+1|>\sqrt{m+1}$ are cyclic vectors for $(W_n)_{n\in\mathbb{N}}$ (Theorem \ref{Cyclic examples}) and end by generalizing Báez-Duarte's theorem (Theorem \ref{GeneralizatioBD2} and Theorem \ref{BD general H2}).

\section{Preliminaries}
\subsection{The Hardy-Hilbert space} We denote by $\mathbb{D}$ and $\mathbb{T}$ the open unit disk and the unit circle respectively. Let $H^{2}$ be the Hilbert space of analytic functions $f(z)=\sum_{n=0}^\infty \hat{f}(n)z^n$ and $g(z)=\sum_{n=0}^\infty \hat{g}(n)z^n$ defined on $\mathbb{D}$ for which the inner product is given by
\begin{equation}\label{norm H}
\langle f,g\rangle:=\sum_{n=0}^\infty \hat{f}(n)\overline{\hat{g}(n)}.
\end{equation}
There also exists an area integral form of the corresponding $H^{2}$-norm given by
\begin{equation}\label{integral norm H}
||f||_{H^2}^{2}=\sup_{0\leq r<1}\frac{1}{2\pi}\int_0^{2\pi}|f(re^{i\theta})|^2d\theta,
\end{equation}
%A holomorphic function $f$ on $\mathbb{D}$ belongs to the Hardy-Hilbert space $H^2$ if
%\[
%||f||_{H^2}=\sup_{0\leq r<1}\left(\frac{1}{2\pi}\int_0^{2\pi}|f(re^{i\theta})|^2d\theta\right)^{1/2}<\infty.
%\]
%The space $H^2$ is a Hilbert space with inner product
%\[
%\langle f,g \rangle=\sum_{n=0}^\infty \hat{f}(n)\overline{\hat{g}(n)},
%\]
where $(\hat{f}(n))_{n\in\mathbb{N}}$ and $(\hat{g}(n))_{n\in\mathbb{N}}$ are the  Maclaurin coefficients for $f$ and $g$ respectively. 
\subsection{Local Dirichlet spaces}
 Let $\zeta\in\mathbb{T}$ and define the local Dirichlet space $\mathcal{D}_{\delta_{\zeta}}$ at $\zeta$ as those $f\in H^{2}$ that satisfy
\begin{equation}\label{Dirichlet integral}
\mathcal{D}_{\zeta}(f)=\int_{\mathbb{D}} \vert f'(z)\vert^{2}\dfrac{1-\vert z\vert^{2}}{\vert z-\delta\vert^{2}}dA(z)<\infty,
\end{equation}
where $dA$ is the normalized area measure of $\mathbb{D}$. The recent monograph \cite{Dirichlet book} contains a comprehensive treatment of local Dirichlet spaces. These subspaces consist of functions of the form 
\begin{equation}\label{Functions in Local D space}
\mathcal{D}_{\delta_{\zeta}}=\lbrace a+(z-\zeta)g(z): g\in H^{2}\ \textup{and}\ a\in\mathbb{C}\rbrace
\end{equation}
(\cite[Theorem 7.2.1]{Dirichlet book}) and contain all functions analytic on a neighborhood of $\overline{\mathbb{D}}$. An alternate way to verify the integral condition \eqref{Dirichlet integral} is as follows. Let $f$ be holomorphic on $\mathbb{D}$ with $f(z)=\sum_{n=0}^\infty a_n z^n$ and $\zeta\in\mathbb{T}$. If both 
\begin{equation}\label{Dirichlet integral series}
\sum_{n=0}^\infty a_n\zeta^n<\infty \ \ \ \mathrm{and} \ \ \ \sum_{k=0}^\infty|\sum_{n=k+1}^\infty a_n\zeta^n|^2<\infty
\end{equation}
 then $\mathcal{D}_\zeta(f)<\infty$ (see \cite[page 115]{Dirichlet book}). In \cite[Theorem 12]{Waleed} it was shown that
\begin{equation}\label{Nperp intesection D}
	\mathcal{N}^{\perp}\cap \mathcal{D}_{\delta_{1}}=\lbrace 0\rbrace.
	\end{equation}
 Keep in mind that $\mathcal{N}^\perp=\{0\}$ is equivalent to the RH (Theorem \ref{Waleedsemigroup}).

\section{The semigroup $\mathcal{W}=(W_n)_{n\in\mathbb{N}}$}

We first establish an adjoint formula for the elements of the semigroup $\mathcal{W}$ which will be used throughout the rest of this work. We omit the straightforward proof. 
\begin{lem}\label{W_n^*} For any $f\in H^2$ and $n\in\mathbb{N}$, the adjoint of $W_n$ is given by
\begin{align*}
W_{n}^{*}f(z)=\sum_{k=0}^\infty B_n(k)z^k,
\end{align*}
where $(\hat{f}(k))_{k\in\mathbb{N}}$ are the  Maclaurin coefficients of $f$ and \[B_n(k)=\hat{f}(nk)+\hat{f}(nk+1)+\dots+\hat{f}(nk+n-1)\]
is the sum of the k-th block of $n$ consecutive coefficients.
\end{lem}

As a consequence of Lemma \ref{W_n^*} it follows that
\begin{equation}\label{Wn isometry}
W_n^*W_n=nI
\end{equation}
for all $n\in\mathbb{N}$, and therefore that $(W_n/\sqrt{n})_{n\in\mathbb{N}}$ is a semigroup of \emph{isometries}. Recall that an isometry $S$ on a Hilbert space $\mathcal{H}$ is called a \emph{shift} if $\Vert S^{*n}f\Vert\rightarrow0$ for all $f\in\mathcal{H}$ and its \emph{multiplicity} is defined as the dimension of $\mathrm{Ker(S^*)}$.
\begin{prop}\label{Wn shift} Each $W_n/\sqrt{n}$ for $n\geq 2$ is a shift with infinite multiplicity and hence $W_n^*$ is universal in the sense of Rota. 
	\end{prop}
\begin{proof} Suppose $T_{\psi,\phi}$ is a weighted composition operator on $H^2$ defined by
	\[
	T_{\psi,\phi}f=\psi (f\circ\phi)
	\] 
	for $f\in H^2$, where $\psi$ is an analytic function on $\mathbb{D}$ and $\phi$ an analytic selfmap of $\mathbb{D}$ with a fixed point $p\in\mathbb{D}$. Then Matache \cite[Theorem 8]{Matache weighted} proved that $T_{\psi,\phi}$ is a shift operator if and only if $|\psi(p)|<1$. Therefore with $\psi(z)=(1+z+\ldots+z^{n-1})/\sqrt{n}$, $\phi(z)=z^n$ and $p=0$, we get $W_n/\sqrt{n}$ is a shift for all $n\in\mathbb{N}$. It is easy to verify that $\mathrm{Ker}(W_{n}^{*})$ contains the functions
	\begin{align*}
	f_k(z)= z^{nk}+z^{nk+1}+\dots+z^{nk+n-2}-(n-1)z^{nk+n-1}
	\end{align*}
	for all $k\in\mathbb{N}$ by Lemma \ref{W_n^*}. Since $(f_k)_{k\in\mathbb{N}}$ is an orthogonal sequence, it follows that $\mathrm{Ker}(W_{n}^{*})$ is infinte-dimensional and hence $W_n/\sqrt{n}$ has infinite multiplicity. The surjectivity of $W_n^*$ follows from \eqref{Wn isometry} and hence each $W_n^*$ is universal in the sense of Rota by the Caradus Criterion.
	\end{proof}

Let $\sigma(T)$  ($\sigma_o(T)$) denote the spectrum (point spectrum) of an operator $T$, and $\overline{B}(a,r)$ the closed ball of radius $r$ centered at $a\in\mathbb{C}$. So we immediately get the following by Proposition \ref{Wn shift}.
\begin{cor} $\sigma(W_{n}^{*})=\sigma(W_{n})=\overline{B}(0,\sqrt{n})$ for all $n\geq 2$. Also each $\sigma_o(W_{n})=\emptyset$ and hence $W_n$ has no finite dimensional invariant subspaces.
\end{cor}
 A proper invariant subspace $E$ for an operator $T$ is \emph{maximal} if it is not contained in any other proper invariant subspace for $T$. In this case $E^\perp$ is a minimal invariant subspace for $T^*$. Hence the ISP may be reformulated in terms of $W_n$. 
\begin{cor}\label{ISP maximal Wn} For any $n\geq 2$, every maximal invariant subspace for $W_n$ has codimension one  if and only if the ISP has a positive solution. 
\end{cor}

Therefore studying the invariant subspaces for $\mathcal{W}$, which we do in the next section, may shed light on both the RH and the ISP.

\section{Invariant subspaces of $\mathcal{W}$ }\label{Properties}
We start by noticing that for $n\geq 1$
\[
W_n=(I+S+\ldots+S^{n-1})T_n
\]
where $T_{n}f(z)=f(z^{n})$ and $S$ is the shift operator on $H^2$ (see Introduction). Since the closure of $\mathcal{N}=\mathrm{span}\{h_k:k\geq 2\}$ is an invariant subspace for each $W_n$, one may also ask whether the closure of $\mathcal{N}$ is a shift-invariant subspace. The complete characterization of shift-invariant subspaces in $H^2$ was provided by Beurling \cite{Beuling shift invariant} and is a landmark results in operator-related function theory. The fascinating answer follows from the fact that $(I-S)\mathcal{N}$ is dense in $H^2$ (see \cite[Theorem 9]{Waleed}).
\begin{prop}\label{Beurling RH} The closure of $\mathcal{N}$ is shift-invariant if and only if the RH holds.
	\end{prop}
\begin{proof} If the RH holds then the closure of $\mathcal{N}$ is $H^2$ which is trivially shift-invariant. The converse follows from the relation $$(I-S)\mathcal{N}\subset (I-S) \overline{\mathcal{N}}\subset \overline{\mathcal{N}}$$ and hence $\overline{\mathcal{N}}=H^2$ by the density of $(I-S)\mathcal{N}$ which implies RH by Theorem \ref{Waleedsemigroup}.
	\end{proof}

The main objective of this section is to consider a special sublattice of the lattice $\mathrm{Lat}(\mathcal{W})$ of common invariant subspaces of the $W_n$ for $n\geq 2$. Define the manifolds
\[
\mathcal{M}:=\mathrm{span}\{h_k-h_\ell:k,\ell\geq 2\} \ \  \ \mathrm{and} \ \  \mathcal{M}_d:=\mathrm{span}\{h_k-h_\ell:k,\ell\in d\mathbb{N}\}
\]
for $d\in\mathbb{N}$ whose closures belong to $\mathrm{Lat}(\mathcal{W})$ due to the identity $W_nh_k=h_{nk}-h_n$ (see \cite[page 249]{Waleed}). It is clear that $\mathcal{M}_{d}\subset\mathcal{M}\subset\mathcal{N}$, and $d_1$ divides $d_2$ if and only if $
\mathcal{M}_{d_2}\subset\mathcal{M}_{d_1}$. It follows that $(\mathcal{M}_d)_{d\in\mathbb{N}}$ is a sublattice of $\mathrm{Lat}(\mathcal{W})$ which is isomorphic to $\mathbb{N}$ with respect to division, and $\mathcal{M}_p$ for any prime $p$ is a maximal element in $(\mathcal{M}_d)_{d\in\mathbb{N}}$. Notice that for any $k,\ell\in d\mathbb{N}$ we have
\[
h_k-h_l=(h_k-h_d)-(h_\ell-h_d)=W_d(h_{k/d}-h_{\ell/d})\in W_d(\mathcal{N})
\]
 (where $h_1\equiv 0$) and therefore $\mathcal{M}_d= W_d(\mathcal{N})$. Since each $W_d$ is a shift of infinite multiplicity, it follows that $\mathcal{M}_d$ has infinite codimension. In particular, the range of $W_d$ is the closure of $\mathcal{M}_d$ if and only if the RH holds by Theorem \ref{Waleedsemigroup}. We shall need a lemma before moving on. Let $\bigvee_n E_n$ denote the smallest closed subspace containing the sets $E_n$. 
\begin{lem}\label{Intersection of kernels} The semigroup $\mathcal{W}$ satisfies
	\[
	\bigcap_{n=k+1}^{\infty}\textup{Ker}W_{n}^{*}=\textup{span}_{1\leq n\leq k}\left\lbrace 1-z^{n}\right\rbrace \ \ \ \forall \ \ k\geq 1.
	\]

\end{lem}
\begin{proof}
	Suppose $f\in\textup{Ker}W_{n}^{*}$ for all $n\geq k+1$. Computing the first term for the power series in Lemma \ref{W_n^*}, we get
	\[
	B_n(0)=\hat{f}(0)+\hat{f}(1)+\dots+\hat{f}(n-1)=0 \ \ \ \forall \ \ n\geq k+1.
	\]
	This implies that $\hat{f}(n)=B_{n+1}(0)-B_n(0)=0$ for all $n\geq k+1$ and hence
	\begin{align*}
	f(z)&=\sum_{n=0}^k\hat{f}(n)z^n=-\hat{f}(1)-\dots-\hat{f}(k)+\sum_{n=1}^k\hat{f}(n)z^n =\sum_{n=1}^k\hat{f}(n)(z^n-1).
	\end{align*}
	Therefore we get $f\in\textup{span}\left\lbrace 1-z,1-z^{2},\dots,1-z^{k}\right\rbrace$ and
	\begin{align*}
	\bigcap_{n=k+1}^{\infty}\textup{Ker}W_{n}^{*}\subset \textup{span}\left\lbrace 1-z,1-z^{2},\dots,1-z^{k}\right\rbrace.
	\end{align*} 
	The opposite inclusion holds since $W_n^*(1-z^k)=0$ for $n> k$ by Lemma \ref{W_n^*}. 
\end{proof}
We arrive at the main result of this section which suggests that the truth of the RH supports the ISP (see Corollary \ref{ISP maximal Wn}).

\begin{thm}\label{4}
	The closure of $\mathcal{M}$ is a proper element in $\mathrm{Lat}(\mathcal{W})$ and equals $\bigvee_{d=2}^{\infty}\mathcal{M}_d$. The closure of $\mathcal{M}$ is maximal in $\mathrm{Lat}(\mathcal{W})$  if and only if the RH is true. In this case we have $\mathrm{codim}(\mathcal{M})=1$.
\end{thm}
\begin{proof} We first show that $\overline{\mathcal{M}}$ equals
	$\bigvee_{n=2}^{\infty}W_n(\mathcal{N})$. 
	Since $W_n(\mathcal{N})\subset\mathcal{M}$ for all $n\geq 2$, it is clear that $\bigvee_{n=2}^{\infty}W_n\mathcal{N}\subset\overline{\mathcal{M}}$. For the opposite inclusion notice that each difference can be written as
	\[
	h_{k}-h_{\ell}=(h_{k\ell}-h_{\ell})-(h_{k\ell}-h_{k})=W_\ell h_k-W_k h_\ell\in\bigvee_{n=2}^{\infty}W_n(\mathcal{N})
	\]
	and hence $\overline{\mathcal{M}}\subset\bigvee_{n=2}^{\infty}W_n(\mathcal{N})$. Therefore  $\overline{\mathcal{M}}=\bigvee_{d=2}^{\infty}\mathcal{M}_d$ since $\mathcal{M}_d= W_d(\mathcal{N})$. To prove that it is a proper subspace notice that
	\begin{equation}\label{M subset (1-z)perp}
	\overline{\mathcal{M}}=\bigvee_{n=2}^{\infty}W_n(\mathcal{N})\subset \bigvee_{n=2}^{\infty}\textup{Im}W_{n}=\left(\bigcap_{n=2}^{\infty}\textup{Ker}W_{n}^{*}\right)^\perp=\{1-z\}^\perp
	\end{equation}
	by Lemma \ref{Intersection of kernels} with $k=1$. Now suppose RH is true and hence $\overline{W_n(\mathcal{N})}= \textup{Im}W_{n}$ by Theorem \ref{Waleedsemigroup}. So we actually have equality in \eqref{M subset (1-z)perp} which implies that $\mathcal{M}^\perp=\mathbb{C}(1-z)$. Hence $\mathrm{codim}(\mathcal{M})=1$ and the closure of $\mathcal{M}$ is maximal for all $W_n$.
Conversely, suppose the RH is false. Then the closure of $\mathcal{N}$ is a proper subspace of $H^2$ and so
\[
\overline{\mathcal{M}}\subset\overline{\mathcal{N}}\subsetneq H^2.
\]
The goal is to prove that $\overline{\mathcal{M}}$ is not maximal for any $W_n$ by showing that $\overline{\mathcal{M}}\neq\overline{\mathcal{N}}$. Note that $1-z$ belongs to $\mathcal{M}^\perp\cap \mathcal{D}_{\delta_{1}}$ by \eqref{Functions in Local D space} and \eqref{M subset (1-z)perp}. But $\mathcal{N}^{\perp}\cap \mathcal{D}_{\delta_{1}}=\left\lbrace 0\right\rbrace$ by \eqref{Nperp intesection D}. This proves that $\overline{\mathcal{M}}\neq\overline{\mathcal{N}}$ and hence the converse.
\end{proof}

\section{Does the closure of $\mathcal{N}$ reduce any $W_n$?}
We know that the closure of $\mathcal{N}$ is invariant under $W_n$ for all $n\in\mathbb{N}$. Is the closure of $\mathcal{N}$ also invariant under any $W_n^*$, equivalently, does the closure of $\mathcal{N}$ \emph{reduce} any $W_n$. We shall need the following consequence of the Wold decomposition for shifts.

\begin{thm}\label{ShiftTheorem}(See \cite[page 4]{Hardy classes})
	Let $S\in B(\mathcal{H})$ be a shift operator. A subspace $M$ of  $\mathcal{H}$ reduces $S$ if and only if $$M=\sum_{j=0}^{\infty}\bigoplus S^{j}E,$$ where $E$ is a subspace of \textup{Ker}$S^{*}$.
\end{thm}
We shall also need the following lemma about kernels of $W_n^*$.
\begin{lem}\label{lemnewreform}
	\textup{Ker} $W_{n}^{*}$ is a subset of the local Dirichlet space $ \mathcal{D}_{\delta_1}$ for all $n\geq 2$.
\end{lem}
\begin{proof}
	Let $f\in\textup{Ker} W_{n}^{*}$. Our goal is to prove that $\mathcal{D}_1(f)<\infty$ using \eqref{Dirichlet integral series}. By Lemma \ref{W_n^*} we have
	\begin{equation}\label{Bn(k)}
	B_n(k)=\hat{f}(nk)+\hat{f}(nk+1)+\dots+\hat{f}(nk+n-1)=0
	\end{equation}
	for all $k\geq 0$. In particular this implies that 
	\begin{equation}\label{sum of Bn(k)}
	\sum_{j=0}^\infty\hat{f}(j)=\sum_{k=0}^\infty B_n(k)=0
	\end{equation}
	 which is the first condition in \eqref{Dirichlet integral series}. For the second condition, let $k_i$ be the unique postive integer such that $nk_i\leq i+1\leq nk_i+n-1$ for each $i\geq 0$. Then by \eqref{Bn(k)} we have
	\begin{align*}
	\mathcal{D}_1(f)&=\sum_{i=0}^\infty|\sum_{j=i+1}^\infty\hat{f}(j)|^2
	=\sum_{i=0}^\infty|\sum_{j=i+1}^{nk_i+n-1}\hat{f}(j)|^2
	\leq \sum_{i=0}^\infty \left(\sum_{j=nk_i}^{nk_i+n-1}|\hat{f}(j)|\right)^2 \\
	&\leq 2^n\sum_{i=0}^\infty \sum_{j=nk_i}^{nk_i+n-1}|\hat{f}(j)|^2=2^n n\sum_{j=0}^\infty |\hat{f}(j)|^2<\infty.
	\end{align*}
	Therefore $f\in\mathcal{D}_{\delta_1}$.
\end{proof}
We are ready to answer the question and again it is equivalent to the RH.
\begin{thm}\label{newreform}
	The closure of $\mathcal{N}$ reduces $W_{n}$ for some $n\geq 2$ if and only if the RH is true.
\end{thm}
\begin{proof}
	Suppose the RH is true. Then the closure of $\mathcal{N} $ is $H^{2}$ which trivially reduces $W_n$ for all $n\geq 2$. Conversely, suppose the closure of $\mathcal{N}$ reduces $W_{n}$ for some $n\geq 2$. This implies that $\mathcal{N}^{\bot}$ also reduces $W_{n}$. Since $W_{n}/\sqrt{n}$ is a shift operator by Proposition \ref{Wn shift}, it follows by Theorem \ref{ShiftTheorem} that there exists a subspace $E$ of \textup{Ker}$W_{n}^{*}$ such that
	\begin{center}
		$\mathcal{N}^{\bot}=\sum_{j=0}^{\infty}\bigoplus W_{n}^{j}E$.
	\end{center}
	In particular $E\subset \mathcal{N}^{\bot}$ and by Lemma \ref{lemnewreform} we have $E\subset\mathrm{Ker}W_n^*\subset\mathcal{D}_{\delta_{1}}$. But since $\mathcal{D}_{\delta_{1}}\cap\mathcal{N}^{\bot}=\left\lbrace 0\right\rbrace$ by \eqref{Nperp intesection D} we have $E=\left\lbrace 0\right\rbrace$. So $\mathcal{N}^{\bot}=\left\lbrace 0\right\rbrace$ and the RH holds. 
\end{proof} 

\section{Cyclic vectors}

Until now the only known example of a cyclic vector for $\mathcal{W}$ is the constant $1$. Indeed $(W_n1)(z)=1+z+\ldots+z^{n-1}$ for all $n\geq 1$ so $\mathrm{span}\{W_n1:n\geq 1\}$ contains all analytic polynomials and is hence dense in $H^2$. Our first objective is to obtain a  new family of cyclic vectors for $\mathcal{W}$. 

\begin{thm}\label{Cyclic examples} For each $m\in\mathbb{N}$ and $\lambda\in\mathbb{C}$ with $|\lambda+1|>\sqrt{m+1}$, the polynomial
	$p_{m,\lambda}(z):=z^{m}+\dots+z-\lambda$ is a cyclic vector for $\mathcal{W}$.
\end{thm} 
\begin{proof}
	Note first that
	\begin{align*}
	W_{n}(z^{m}+\dots+z-\lambda)=\sum_{j=n}^{(m+1)n-1}z^{j}-\lambda \sum_{j=0}^{n-1}z^{j}.
	\end{align*}
	Now suppose there exists an $f\in H^{2}$ such that $f\perp W_{n}p_{m,\lambda}$, for every $n\geq1$. Looking first at the particular cases $n=(m+1)^{k}$, we get:
	\begin{align*}
	k &=0         &  \sum_{j=1}^m\hat{f}(j) & = \lambda\hat{f}(0)  \\
	k &=1         &    \sum_{j=m+1}^{(m+1)^2-1}\hat{f}(j) & = \lambda \sum_{j=0}^m\hat{f}(j) =\lambda(\lambda+1)\hat{f}(0)  \\
	k &=2         &  \sum_{j=(m+1)^2}^{(m+1)^3-1}\hat{f}(j) & =\lambda\sum_{j=0}^{(m+1)^2-1}\hat{f}(j) =\lambda(\lambda+1)^{2}\hat{f}(0)
	\end{align*}
	and for the general case we use induction to get
	\begin{align*}
	\sum_{j=(m+1)^{k}}^{(m+1)^{k+1}-1}\hat{f}(j)=\lambda(\lambda+1)^{k}\hat{f}(0).
	\end{align*}
	By the geometric sum formula with common ration $\lambda+1$ we get
	\begin{align*}
	\sum_{j=0}^{(m+1)^{k+1}-1}\hat{f}(j)&
	=\sum_{j=0}^k\lambda(\lambda+1)^{j}\hat{f}(0)+\hat{f}(0)\\&
	=\left(\dfrac{1-(\lambda+1)^{k+1}}{1-(\lambda+1)}\right)\lambda\hat{f}(0)+\hat{f}(0)\\&
	=(\lambda+1)^{k+1}\hat{f}(0).
	\end{align*}
	Using the Cauchy-Schwarz inequality, we get
	\[| \hat{f}(0)|\leq\left( \dfrac{\sqrt{m+1}}{|\lambda+1|}\right) ^{k+1}\left\| f\right\| \ \ \forall \ \ k\in\mathbb{N}\]
	and $\sqrt{m+1}<|\lambda+1|$ implies that $\hat{f}(0)=0$. By induction if we assume \[\hat{f}(0)=\ldots=\hat{f}(s-1)=0\] for any $s\in\mathbb{N}$, we can use the same argument above taking $n=(s+1)(m+1)^{k}$ to show that $\hat{f}(s)=0$ and thus $f\equiv0$. Therefore, $p_{m,\lambda}(z)=z^{m}+\dots+z-\lambda$ is a cyclic vector for $|\lambda+1|>\sqrt{m+1}$.
\end{proof}
Therefore we may replace the constant $1$ in Theorem \ref{waleed} by any of the polynomials above by Theorems \ref{Waleedsemigroup}.
\begin{cor}\label{concrete cyclic vectors} The RH holds if and only if $z^{m}+\dots+z-\lambda$ belongs to the closure of $\mathcal{N}$ for any $|\lambda+1|>\sqrt{m+1}$.
	\end{cor}
Notice that $\lambda$ cannot be equal to $-1$ in the polynomials above. That is because they are no longer cyclic vectors for $\mathcal{W}$ as shown by the following result.
\begin{prop}\label{non-cyclicity}
	If $f\in H^2$  is a cyclic vector for $\mathcal{W}$, then $\hat{f}(0)\neq\hat{f}(1)$.
\end{prop}
\begin{proof} The condition $\hat{f}(0)=\hat{f}(1)$ is equivalent to 
	$f\perp(1-z)$. For any $f\in H^{2}$ 
	\begin{align*}
	\overline{\textup{span}\left\lbrace W_{n}f:n\geq2\right\rbrace }\subset\bigvee_{n=2}^{\infty}\text{Im}W_{n}=\left(\bigcap_{n=2}^{\infty}\textup{Ker}W_{n}^{*}\right)^\perp=\left\lbrace 1-z\right\rbrace ^{\perp}
	\end{align*} 
	by Lemma \ref{Intersection of kernels}. In particular, if $f\perp(1-z)$ then
	\begin{align*}
	\overline{\textup{span}\left\lbrace W_{n}f:n\geq1\right\rbrace }\subset\left\lbrace 1-z\right\rbrace ^{\perp}\subsetneq H^{2}
	\end{align*}
	since $W_1\equiv I$. This proves that $f$ is not cyclic.
\end{proof}

Finally, we show that there exist polynomials which are \emph{not} cyclic for $\mathcal{W}$ but for which the conclusion of Corollary \ref{concrete cyclic vectors} still holds. Yang \cite{Yang} gave a generalization of the classical Nyman-Beurling reformulation for the RH, replacing $\chi_{(0,1)}$ with $\chi_{(a,b)}$ for any $0\leq a<b\leq1$. Recall that $f_{\lambda}(x)=\{\lambda/x\rbrace-\lambda\lbrace 1/x\}$ for $0<\lambda\leq 1$. Yang proved that the RH is true if and only if $\chi_{(a,b)}$ belongs to the closed linear span of $\left\lbrace f_{\lambda}:0<\lambda\leq1\right\rbrace,$ for any $0\leq a<b\leq1.$ In \cite{Bagchi} it was shown that the Baéz-Duarte criterion is equivalent to the density of the linear span of $\left\lbrace f_{1/k}:k\geq1\right\rbrace$ in the closed subspace $\mathcal{V}$ of $L^{2}(0,1]$, consisting of the functions almost everywhere constant on the sub-intervals $(\frac{1}{n+1},\frac{1}{n}]$ for $n\geq 1$. Combining these two results  gives a generalization of  Baéz-Duarte's original theorem in $L^2(0,1]$.

\begin{thm}\label{GeneralizatioBD2}
	The RH is true if and only if $\chi_{(\frac{1}{m},\frac{1}{n}]}$ belongs to the closed linear span of $\left\lbrace f_{1/k}:k\geq2\right\rbrace$ in $L^{2}(0,1]$ for any $n,m\in\mathbb{N}$ with $n<m\leq\infty$.
\end{thm}

Now consider the unitary operator $\Phi=T^{-1}\circ\Psi: \mathcal{V}\rightarrow H^{2}$, where $\Psi:\mathcal{V}\rightarrow \mathcal{A}$ and $T:H^{2}\rightarrow \mathcal{A}$ are unitary operators defined by
\[
	(\Psi f)(z)=\sum_{k=0}^{\infty}f\left(\frac{1}{k+1}\right) z^{k} \ \ \mathrm{and} \ \ 	(Th)(z)=\dfrac{((1-z)h(z))'}{1-z}
\]
and $\mathcal{A}$ is a particular weighted Bergman space (see \cite{Waleed}). Transferring Theorem \ref{GeneralizatioBD2} with $m=\infty$ to the Hardy space $H^2$ via the map $\Phi$ gives the desired polynomials that are non-cyclic for $\mathcal{W}$ by Proposition \ref{non-cyclicity}. 

\begin{thm}\label{BD general H2}  The RH is true if and only if $1+z+\dots+z^{n}$ belongs to the closure of $\mathcal{N}$ in $H^{2}$ for any $n\in\mathbb{N}$.
	\end{thm}
\begin{proof} This follow by the fact that $\Phi f_{1/k}=h_k$ for all $k\geq 2$ (see \cite{Waleed}) and a simple computation shows $\Phi \chi_{(0,\frac{1}{n+1}]}$ is a scalar multiple of $1+z+\dots+z^{n}$ for all $n\in\mathbb{N}$.  
	\end{proof}

\section*{Acknowledgements}
This work was financed in part by the Coordenação de Aperfeiçoamento de Nivel Superior- Brasil (CAPES)- Finance Code 001. It constitutes a part of the doctoral thesis of the first and third authors, under the supervision of the second author.
\bibliographystyle{amsplain}

\begin{thebibliography}{00}
	\bibitem{Bagchi} B. Bagchi, On Nyman, Beurling and B\'{a}ez-Duarte's Hilbert space reformulation of the Riemann hypothesis, Proc. Ind. Acad. Sci (Math. Sci.), 116(2), 137-146, 2006. 
\bibitem{Beuling shift invariant} A. Beurling, On two problems concerning linear transformations in Hilbert spaces, Acta Math., 81 (1949), 239-253.





\bibitem{Beurling} A. Beurling, A closure problem related to the Riemann zeta-function, Proc. Nat. Acad. Sci., 41, 312-314, 1955.

\bibitem{Beurling 2} A. Beurling, On the completeness of $\psi(nt)$ on $L^2(0,1)$, in Harmonic Analysis, Contemp. Mathematicians, The collected works of Arne Beurling, vol. 2, Birkhauser, Boston, 1989, p. 378-380.

\bibitem{Baez-Duarte} L. B\'{a}ez-Duarte, A strengthening of the Nyman-Beurling criterion for the Riemann hypothesis, Atti Acad. Naz. Lincei 14, 5-11, 2003.


\bibitem{Caradus} S. R. Caradus, Universal operators and invariant subspaces, Proc. Amer. Math. Soc. 23(1969), 526-527. 
%\bibitem{Constara-Ransford} C. Constara, T. Ransford, Which de Branges-Rovnyak spaces are Dirichlet spaces (and vice versa)?. J. Funct. Anal. 265(12), 3204-3218 (2010)

%\bibitem{Hb book vol 1} E. Fricain, J. Mashreghi, The theory of $\mathcal{H}_{b}$ spaces. Vol. 1, volume 20 of New Mathematical Monographs. Cambridge University Press, Cambridge (2016)

%\bibitem{Hb book vol 2} E. Fricain, J. Mashreghi, The theory of $\mathcal{H}_{b}$ spaces. Vol. 2, volume 21 of New Mathematical Monographs. Cambridge University Press, Cambridge (2016)

\bibitem{Chalendar Partington book} I. Chalendar and J. R. Partington, Modern approaches to the Invariant Subspace Problem, Cambridge University Press, 2011.

\bibitem{Hedenmalm-Seip}  H. Hedenmalm, P. Lindqvist, and K. Seip. A Hilbert space of Dirichlet series and systems of
dilated functions in $L^2(0, 1)$, Duke Math. J., 86:1–37, 1997. MR 99i:42033

\bibitem{Hui-Guo} H. Dan and K. Guo, The periodic dilation completeness problem: cyclic vectors in the Hardy space over the infinite-dimensional polydisk, J. Lond. Math. Soc. (2) 103 (2021), no. 1, 1–34. 

\bibitem{Dirichlet book}  J. Mashreghi, K. Kellay, Omar El-Fallah, and T. Ransford, A primer on the Dirichlet space. Cambridge Tracts in Mathematics (203), Cambridge University Press, 2014.

%\bibitem{Montgomery-Vaughan} H. L. Montgomery, R. C. Vaughan, Multiplicative number theory: 1. Classical theory. Cambridge Studies in Advanced Mathematics (97), Cambridge University Press, 2006.

\bibitem{Matache weighted} V. Matache, Isometric weighted composition operators, New York J. Math. 20 (2014) 711–726.

\bibitem{Nikolski} N. Nikolski, In a shadow of the RH: cyclic vectors of the Hardy spaces on the Hilbert multidisc. Ann. Inst. Fourier, 62(5), 1601-1626 (2012).

\bibitem{Nyman} B. Nyman, On some groups and semigroups of translations. Thesis, Uppsala, 1950.


%\bibitem{unbounded toplitz} D. Sarason, Unbounded Toeplitz operators. Integr. Equ. Oper. Theory. 61(2), 281-298 (20
\bibitem{Hardy classes} M. Rosenblum and J. Rovnyak, Hardy classes and operator theory, Courier Corporation, 1997.

\bibitem{Rota} G. C. Rota, On models for linear operators, Comm. Pure Appl. Math. 13 (1960), 469-472.



\bibitem{Waleed} S. Waleed Noor, A Hardy space analysis of the Báez-Duarte criterion for the RH, Adv. Math. 350 (2019), 242-255.



\bibitem{Wintner} A. Wintner, Diophantine approximation and Hilbert's space. Amer. J. Math. 66 (1944), p.564-578.

\bibitem{Yang} J. Yang, A generalization of Beurling's criterion for the Riemann hypothesis, J. Number Theory. 164 (2016) 299-302. 


\end{thebibliography}

\end{document}